\documentclass[12pt]{amsart}
\usepackage[dvips]{graphicx}
\usepackage{color}

 \usepackage{amsfonts, amssymb,verbatim,amsmath, latexsym, textcomp,amscd}

\def\CC{\mathbb C}
\def\Chat{\hat {\mathbb C}}
\def\EE{\mathbb E} 
 
\def\HH{\mathbb H}
\def\NN{\mathbb N}

\def\QQ{\mathbb Q}
\def\Qhat {\hat {\mathbb Q}}
\def\RR{\mathbb R}

\def\ZZ{\mathbb Z}

\def\bb{\beta}

\def\a{\alpha}
\def\b{\beta}

\def\d{\delta}
\def\e{\epsilon}
\def\g{\gamma}

\def\q{q}

\def\s{\sigma}

\def\w{\eta}

\def\bu{{\bf u}}
\def\bv{{\bf v}}
\def\bw{{\bf w}}
\def\bx{{\bf x}}
\def\by{{\bf y}}
\def\bz{{\bf z}}

\def\cal{\mathcal}

\def\B{{ \mathcal B}}
\def\C{{ \mathcal C}}

\def\E{{\mathcal E}}
\def\F{{  \mathcal F}}

\def\H{{\cal H}}

\def\M{{\cal M}}

\def\P{{\cal P}}

\def\S{{\mathcal S}}

\def\R{{\mathcal R}}

\def\T{{\cal T}}

\def\W{{\mathcal W}}

   \def\Ax{\mathop{\rm{Ax}}}

    \def\br{\mathop{\rm \bf{br}}}

    \def\ch{\mathop{\rm{cosh}}} 
     \def\sh{\mathop{\rm{sinh}}} 

   \def\Am{\mathop{{ Am}} }

  \def\SL{SL(2,\mathbb C)}

  \def\Tr{\mathop{\rm{Tr}}}

\def\dd{\partial}

\def\square{\hfill${\vcenter{\vbox{\hrule height.4pt \hbox{\vrule width.4pt
height7pt \kern7pt \vrule width.4pt} \hrule height.4pt}}}$}

\def\co{{\colon \thinspace}}

\newtheorem{theorem}{Theorem}[section]
\newtheorem{definition}[theorem]{Definition}
\newtheorem{lemma}[theorem]{Lemma}
\newtheorem{proposition}[theorem]{Proposition}
\newtheorem{corollary}[theorem]{Corollary}
\newtheorem{remark}[theorem]{Remark}

\newtheorem{introthm}{Theorem}

\begin{document}
\title{Primitive stability and the Bowditch conditions revisited}

 \author{Caroline Series}

\address{\begin{flushleft} \rm {\texttt{C.M.Series@warwick.ac.uk \\http://www.maths.warwick.ac.uk/$\sim$masbb/} }\\ Mathematics Institute, 
 University of Warwick \\
Coventry CV4 7AL, UK
\end{flushleft}}

 \begin{abstract}
The equivalence of two conditions on the primitive elements in an $SL(2,\mathbb C)$ representation of the free group $F_2 = <a,b>$, namely  Minsky's condition of primitive stability and the $BQ$-conditions  introduced by Bowditch and generalised by Tan, Wong  and Zhang, has been proved by Lee and Xu  and independently by the author in arXiv:1901.01396.   This note is a revised version of our original proof, which is greatly simplified by incorporating some of the ideas introduced by Lee and Xu, combined with the language of the Bowditch tree.

 {\bf Keywords: Free group on two generators, Kleinian  group, non-discrete representation, palindromic generator, primitive stable} \\
 
\centering{  \emph{ To Ser Peow Tan on his 60th birthday.} }

\end{abstract}

\ 
 \date{\today}
\maketitle
 
\noindent {\bf MSC classification:}    {30F40 (primary), 57M50 (secondary).}

 \section{Introduction}

 In this note we show the equivalence of two conditions on the primitive elements in an $\SL$ representation $\rho$ of the free group  
$F_2 = <a,b>$ on two generators, which may hold even when the image $ \rho(F_2)$ is not discrete. One is the condition of primitive stability $PS$ introduced by Minsky~\cite{Minsky} and the other is the so-called $BQ$-conditions  introduced by Bowditch~\cite{bow_mar} and generalised by Tan,  Wong  and Zhang~\cite{tan_gen}.
This result was proved in~\cite{LX} and independently in~\cite{serPS}.   This note is a revised version of 
~\cite{serPS}, which  can be greatly simplified by incorporating the elegant estimates and ideas in~\cite{LX}. The reason for writing it is to give a concise presentation using the language of the Bowditch tree developed in~\cite{bow_mar} and~\cite{tan_gen} and used in~\cite{serPS}.  

Both~\cite{LX} and~\cite{serPS}  introduced a third condition which we call the \emph{bounded intersection property} $BIP$, which they showed was implied by but does not imply the other two. We  also explain this condition and prove the implication here.

We begin by explaining these three conditions one by one.
Recall that an element $ u \in F_2$ is called \emph{primitive} if it forms one of a generating pair $(u,v)$ for $F_2$. Let $\P$ denote the set of primitive elements in $F_2$.  It is well known that up to inverse and  conjugacy, the primitive elements are enumerated by the rational numbers $\Qhat = \QQ \cup \infty$, see Section~\ref{farey} for details.

\subsection{The primitive stable condition $PS$}

The notion of primitive stability was introduced by Minsky in~\cite{Minsky} in order to construct an $Out(F_2)$-invariant subset of the $\SL$ character variety $\chi(F_2)$ strictly larger than the set of discrete free representations.

Let $d(P,Q)$ denote the  hyperbolic distance  between points $P, Q$
in hyperbolic $3$-space $\HH^3$. 
Recall that a path $t \mapsto \gamma(t) \subset \HH^3$  for $t \in I$ (where $I$ is a possibly infinite interval in $ \RR$) is called a $(K, \epsilon)$-quasigeodesic if there exist constants $K, \epsilon >0 $ such that 
\begin{equation} \label{qgeod} K^{-1}|s-t| - \epsilon \leq d(\gamma(s), \gamma(t)) \leq K|s-t| + \epsilon  \ \ \mbox{\rm {for all}} \ \  s, t \in I. \end{equation}

For  a  representation $\rho \co F_2 \to \SL$, in general we will denote elements in $ F_2$ by  lower case letters and their images under $\rho$ by the corresponding upper case, thus $X = \rho (x)$ for $x \in F_2$. In particular if $(u,v)$ is a generating pair  for  $F_2$  we write $U = \rho(u), V = \rho(v)$.  

Fix once and for all a basepoint $O \in \HH^3$ and suppose that $w= e_{1} \ldots e_{n} , e_{k} \in  \{u^{\pm}, v^{\pm} \} , k = 1, \ldots, n$ is a cyclically shortest word in the generators $(u,v)$.
The \emph{broken geodesic} $\br _{\rho}(w; (u,v))$  of $w$ with respect to  $(u,v)$  is the infinite path of geodesic segments joining vertices $$\ldots,  \ , E_n^{-1}E_{n-1}^{-1}  E_{n-2}^{-1}O, E_n^{-1}E_{n-1}^{-1}  O, E_n^{-1}O, O, E_1O, E_1 E_2 O,  \ldots,  E_1E_2 \ldots E_n  O, E_1E_2 \ldots E_n E_1 O, \ldots. $$
where $E_i = \rho(e_i)$.   
\begin{definition} \label{definePS}
Let $(u,v)$ be a fixed generating pair for $F_2$. A representation $\rho \co F_2 \to \SL$ is \emph{primitive stable}, denoted $PS$,  if the broken geodesics $\br_{\rho} (w; (u,v))$  for all  words $w= e_{1} \ldots e_{n}  \in \P, e_{k} \in \{ u^{\pm},v^{\pm}\}, k = 1, \ldots, n$,  are  uniformly $(K,\e)$-quasigeodesic for some fixed constants $(K,\e)$.\end{definition}
Notice that this definition is independent of the choice of basepoint $O$ and makes sense since  the change
from 
$\br_{\rho} (w; (u,v))$ to $\br_{\rho} (w; (u',v'))$ for some other generator pair $(u',v')$ 
changes all the constants for all the quasigeodesics uniformly.   

For   $g \in F_2$ write $||g||$  or more precisely $||g||_{u,v}$ for the word length of $g$, that is the shortest representation of $g$ as a  product of generators $(u,v)$.
It is easy to see that for fixed generators, the condition  $PS$ is equivalent to the existence of $K, \epsilon >0 $ such that 
\begin{equation} \label{qgeod1}
K^{-1}||g'|| - \epsilon \leq d(O, \rho(g')O) \leq K||g'|| + \epsilon \end{equation}
for all  finite subwords $g'$ of the infinite reduced word    
$ \ldots e_{1} \ldots e_{ n} \ldots e_{ 1} \ldots e_{ n}\ldots $

Recall that an  irreducible representation $\rho \co F_2 \to \SL$  is determined up to conjugation by the traces of $U = \rho(u),V = \rho(v)$ and $UV= \rho(uv)$ where $(u,v)$ is a generator pair for $F_2$. More generally, if we take the GIT quotient of  all (not necessarily irreducible) representations, then the resulting $SL(2,\CC)$ character variety of $F_2$ can be identified with $\CC^3$ via these traces, see for example \cite{goldman2} and the references therein.   (The only non-elementary (hence reducible) representation occurs when
$\Tr [U,V] = 2$. We  exclude this from the discussion, see for example~\cite{sty} Remark 2.1.)

\begin{proposition}[\cite{Minsky} Lemma 3.2] \label{prop:psopen} The set of primitive stable $\rho \co F_2 \to \SL$ is open  in the $\SL$ character  variety of $F_2$.  \end{proposition}

 Minsky showed that not all $PS$ representations are discrete.

\subsection{The Bowditch $BQ$-conditions}

The $BQ$-conditions were introduced by Bowditch in~\cite{bow_mar} in order to give a purely combinatorial proof of McShane's identity.

Again let $(u,v)$ be a generator pair for $F_2$ and  let $\rho \co F_2 \to \SL$. 

\begin{definition} \label{defineBQ}
Following~\cite{tan_gen}, an irreducible  representation $\rho \co F_2 \to \SL$ is said to satisfy the $BQ$-conditions if 
\begin{equation}  \label{eqn:B2}
 \begin{split}  & \Tr \rho(g) \notin [-2,2]  \ \ \forall g \in \P    \ \ \mbox {\rm and}   \ \ \cr
& \{ g \in \P: |\Tr \rho(g)| \leq 2 \} \ \mbox {\rm is finite}.\end{split}\end{equation}
\end{definition}

We denote the set of all representations satisfying the $BQ$-conditions by $\B$.

\begin{proposition}[\cite{bow_mar} Theorem 3.16, \cite{tan_gen} Theorem 3.2] \label{prop:bqopen}  The set    $\B$ is open in the $\SL$ character  variety of $F_2$.    \end{proposition}

Bowditch's original work~\cite{bow_mar} was on the case in which  the commutator $[X,Y] = XYX^{-1}Y^{-1}$
is parabolic and $\Tr  [X,Y] =  -2 $.  He conjectured that all  representations in $\B$ of this type are quasifuchsian and hence discrete. While this question remains open, it is shown  in~\cite{sty} that without this restriction, there are definitely  representations in $\B$ which are not discrete.

\subsection{The bounded intersection property BIP}\label{sec:BIP}

Recall that a word $w= e_1e_2 \ldots e_n$ in generators $(u,v)$  of $F_2$ is \emph{palindromic}  if  it reads the same forwards and backwards, that is, if $  e_1e_2 \ldots e_n = e_ne_{n-1} \ldots e_1$. Palindromic words have been studied by 
Gilman and Keen in~\cite{gilmankeen1, gilmankeen2}.

Suppose that $\rho \co F_2 \to \SL$ and let $(u,v)$ be a generating pair. Denote  the extended common perpendicular of the axes of $U = \rho(u), V = \rho(v)$ by   $\E(U,V)$. 
By applying the $\pi$ rotation about $\E(U,V)$, it is not hard to see  that if a word $w$ is palindromic in a generator pair $(u,v)$  then  the axis of $W = \rho(w)$ intersects   $\E(U,V)$ perpendicularly, see for example~\cite{BSeries}. (See~\cite{LX} for an interesting remark on the failure of the converse.)

Fix  generators $(a,b)$ for $F_2$. We call the pairs $(a,b), (a,ab)$ and $(b, ab)$ the  \emph{basic generator pairs}.  
Now given $\rho \co F_2 \to \SL$ let $A = \rho(a), B = \rho(b)$ and consider the  three common perpendiculars $\E(A,B), \E(A,AB)$ and $\E(B,AB)$. 
 (We could equally well  chose  to use $BA$ in place of $AB$; the main point is that the choice is fixed once and for all.) 
We call  these lines the \emph{special hyperelliptic axes}.

\begin{definition} \label{defineBIP} 
Fix a basepoint $O \in \HH^3$. A representation 
$\rho \co F_2 \to \SL$ satisfies the \emph{bounded intersection property} $BIP$  if there exists  $D>0$ so that  
if a generator $w$ is palindromic with respect to one of the three basic generators pairs, then its axis intersects the corresponding special hyperelliptic axis in a point at distance at most $D$ from $O$. Equivalently, the axes of all palindromic primitive elements intersect the appropriate hyperelliptic axes in bounded intervals.
\end{definition}
Clearly this definition is independent of the choices of $(a,b)$ and $O$. 

A similar condition but related to \emph{all} palindromic axes was used in~\cite{gilmankeen2} to give a condition for discreteness of geometrically finite groups.

In Section~\ref{BIP} we show that every generator is conjugate to one which is palindromic with respect to one of the three basic generator pairs. In fact each primitive element can be conjugated (in different ways) to be palindromic with respect to two out of the three possible basic pairs. For a more precise statement see Proposition~\ref{uniquepalindromes}.

\subsection{The main result} 

The main results  of this paper are:
\begin{introthm} \label{introthmA} The conditions $BQ$ and $PS$ are equivalent.  \end{introthm}

\begin{introthm} \label{introthmB} The conditions $BQ$ and $PS$ are both imply, but are not implied by, the condition $BIP$. \end{introthm}

 In the case of real representations, Damiano Lupi~\cite{lupi} showed by case by case analysis following \cite{goldman} that the conditions $BQ$ and $PS$ are equivalent.
 
 To see that $BIP$ does not imply the other conditions, first note  that conditions $PS$ and $BQ$ both imply that no element in $\rho(\P)$ is elliptic or parabolic. The condition $BIP$ rules out parabolicity (consider the fixed point of a palindromic parabolic element to be a degenerate  axis which clearly meets the relevant  hyperelliptic axis at infinity). However  the condition does not obviously rule out elliptic elements in $\rho(\P)$.
In particular, consider any  $SO(3)$ representation, discrete or otherwise. Here all axes are elliptic and all pass through a central fixed point which is also at the intersection of all three hyperelliptic axes.  Such a representation clearly satisfies $BIP$.

\begin{remark} \label{binbin} {\rm We remark that the second statement of Theorem II in~\cite{LX} is false: there are $F_2$ representations with discrete image which have property $BIP$ but  which are not in $\B$, for example the finite orthogonal group  consisting of order two rotations round three mutually perpendicular axes. This group clearly satisfies $BIP$ but all its elements are elliptic. (The error stems from an oversight about accumulation points in their proof.) } 
\end{remark}

The plan of the paper is as follows. The hardest part of the work is to prove Theorem~\ref{BQimpliesPS}, that if $\rho$ satisfies the $BQ$-conditions then $\rho$ is primitive stable. In~\cite{serPS} this was done by first showing that  if $\rho$ satisfies the $BQ$-conditions then $\rho$ has the bounded intersection property, and using this to deduce $PS$.  However, as explained in  Section~\ref{geometry},  this is shown to be unnecessarily complicated by the improved estimates and methods of~\cite{LX}.

 In Section~\ref{farey} we present background on the Farey tree and also introduce Bowditch's condition of Fibonacci growth. In Section~\ref{Bowditchbackground}, we summarise Bowditch's method of assigning an orientation to  the edges of the Farey tree ($T$-arrows) and, subject to the $BQ$-conditions,  the existence of a finite attracting subtree. In~\ref{Warrows} we introduce a second way of orientating edges based on word length ($W$-arrows), and show that for all but finitely many words these two orientations coincide.

In Section~\ref{geometry} we collect the background and estimates  used to prove Theorem~\ref{introthmA}. This is based almost entirely on~\cite{LX}, in particular we need the amplitude of a right angle hexagon whose three alternate sides correspond to  the axes of a generator triple $(u,v,uv)$. As we shall explain, this quantity defined in~\cite{Fen} is an invariant of the representation $\rho$ and plays a crucial  part what follows. 
We then  continue following~\cite{LX} to get the crucial result  Proposition~\ref{longwordsqgeod}.  

Theorem~\ref{introthmA} is proved in Section~\ref{sec:BQimpliesBIP}. That $PS$ implies $BQ$ follows easily from the condition of Fibonacci growth (see Definition~\ref{fibonaccidefn}). This was proved  in~\cite{lupi}. 
 Proposition~\ref{longwordsqgeod}   and the results of  Section~\ref{Bowditchbackground} then lead to the proof of Theorem~\ref{BQimpliesPS},  that $BQ$ implies $PS$.

In Section~\ref{BIP} we discuss the condition $BIP$. We begin with a result which may be of independent interest on the palindromic representation of primitive elements, Proposition~\ref{uniquepalindromes}.  Theorem~\ref{introthmB}, that $BQ$ implies $BIP$, is then easily deduced from Theorem~\ref{BQimpliesPS}. In Theorem~\ref{direct}  we give an alternative  direct  proof using Equation~\eqref{distbound}, which uses the invariance of the amplitude of $\rho$ to give an improved version of the estimates  in~\cite{serPS}.

We would like to thank Tan Ser Peow  and Yasushi Yamashita for initial discussions about the original version~\cite{serPS} of this paper.  The work involved in Lupi's thesis~\cite{lupi}  also made a significant contribution.  We also thank Tan for pointing us to the work of Lee and Xu, and for a careful reading of this paper. The idea  of introducing the condition $BIP$ arose while  trying to interpret  some very interesting computer graphics involving non-discrete groups made by Yamashita. We hope to return to this topic elsewhere. 

As we hope we have made clear above, there is little in this revised version of~\cite{serPS} which is not essentially contained in~\cite{LX} and we wish to fully acknowledge the elegance and ingenuity of their method.

 \section{Primitive elements, the Farey tree and Fibonacci growth}  \label{farey} The Farey tessellation $\F$ as shown in Figures~\ref{fig:farey} and \ref{fig:colouredtree}  consists of  the images of the ideal triangle with vertices at $1/0,0/1$ and $1/1$ under the action of $SL(2,\ZZ)$ on the upper half plane, suitably conjugated to the position shown in the disk. The label $p/q$ in the disk is just the conjugated image of the actual point $p/q \in \RR$.

\begin{figure}[ht]
\includegraphics[width=5.5cm]{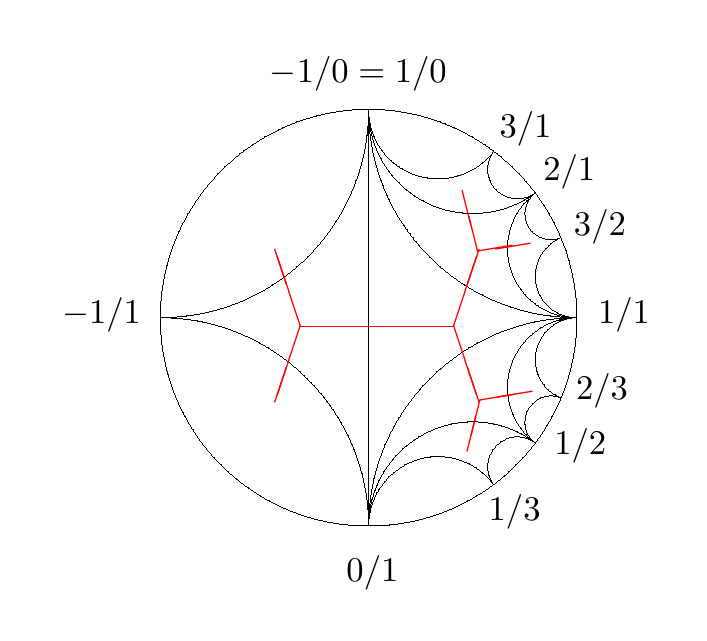}
\hspace{1cm}
\includegraphics[width=5.5cm]{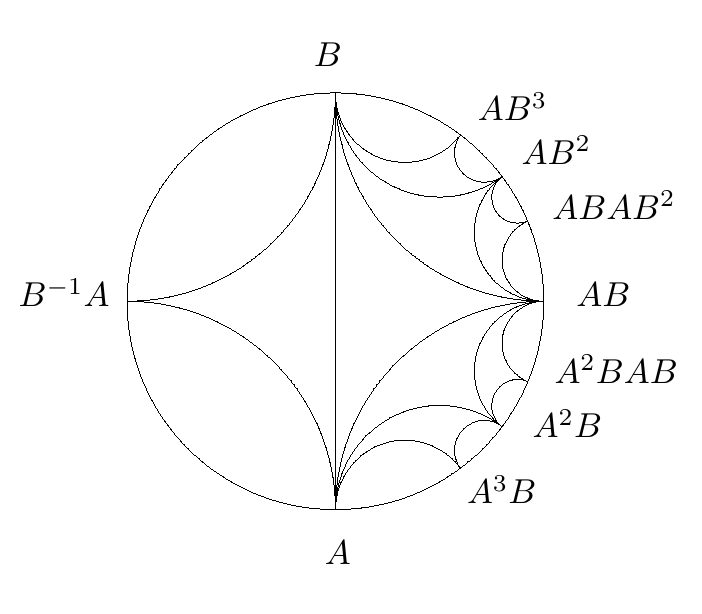}
\caption{The Farey diagram, showing the arrangement of rational numbers on the left with the corresponding  primitive words on the right. The dual graph shown on the left is the Farey tree $\T$. }\label{fig:farey}
\end{figure}

  Since the rational points in $\Qhat = \QQ\cup \infty$ are precisely the images of $\infty$ under $SL(2,\ZZ)$, they correspond bijectively to  the vertices of $\F$.
A pair  $p/q , r/s \in \hat \QQ$ are the endpoints of an edge  if and only if $pr-qs = \pm 1$; such pairs are called \emph{neighbours}.
A triple of points in $\hat \QQ$ are the vertices of a triangle precisely when they are the images of the vertices of the initial triangle $(1/0,0/1,1/1)$; such triples are always of the form 
$(p/q , r/s,(p+r )/( q+s))$ where  $p/q , r/s$ are neighbours.
In other words,  if $p/q , r/s$ are the endpoints of an edge, then the vertex of the triangle on the side away from the centre of the disk is found by `Farey addition' to be $(p+r )/( q+s)$. Starting from $1/0 = -1/0=  \infty$ and $0/1$, all   points in $\hat \QQ$ are obtained recursively  in this way. Note we need to start with $-1/0=  \infty$  to get the negative fractions on the left side of the left hand diagram in Figure~\ref{fig:farey}.

As noted in the introduction, up to inverse and conjugation, the equivalence classes of primitive elements  in $F_2$ are enumerated by $\Qhat$.
Formally, we set $\overline \P $ to be the set of equivalence classes of cyclically shortest primitive elements under the relation  $u \sim v$ if and only if either $v =  gug^{-1}$ or $v = gu^{-1}g^{-1}, g \in F_2$. We call the equivalence classes,  \emph{extended conjugacy classes} and denote the equivalence class of $u \in \P$ by $\bu$. In particular, the set of all cyclic permutations of a given word are in the same extended class. Since we are working in the free group, a word is  \emph{cyclically shortest} if it, together with all its cyclic permutations,  is reduced, that is, contains no occurrences of $x$ followed by $x^{-1}, x \in \{a^{\pm}, b^{\pm}\}$.  

The right hand picture in Figure~\ref{fig:farey}
 shows an enumeration of representative elements from $\overline \P $, starting with initial triple $(a,b,ab)$. Each vertex is labelled by a certain cyclically shortest generator $w_{p/q}$.  Corresponding to the process of Farey addition, the words  $w_{p/q}$ can be found by juxtaposition as indicated on the diagram.  Note that for this to work it is important to preserve the order:   if $u,v$ are the endpoints of an edge  with $u$ before $v$ in the anti-clockwise order round the circle, the correct  concatenation is $uv$, see Figure~\ref{fig:regionboundary}. Note also that the words on the left side of the diagram involve $b^{-1}$ and $a$, rather than $b$ and $a$,  corresponding to starting with $\infty = -1/0 $. 
It is not hard to see that pairs of primitive elements form a generating pair if and only if they are at the two endpoints of an edge of the Farey tessellation, while the words at the vertices of a triangle correspond to a generator triple of the form $(u,v,uv)$.

 The word $w_{p/q}$ is a representative of the extended conjugacy class identified with $p/q \in \hat \QQ $. It is almost but not exactly the same as  the \emph{Christoffel word} as described~\cite{LX}.  We denote this class by $ [p/q]$ and call $w_{p/q}$ the \emph{standard representative} of $ [p/q]$.    Likewise if $p/q, r/s \in \Qhat$ are neighbours we call   $(w_{p/q},w_{r/s})$  the standard (unordered) generator pair.
  It is easy to see that $e_a(w_{p/q}) / e_b(w_{p/q}) = p/q$, where $e_a(w_{p/q}),  e_b(w_{p/q})$ are the sum of the exponents in $w_{p/q}$ of $a,b$ respectively.  All other words in $ [p/q]$ are cyclic permutations of $w_{p/q}$ or its inverse.  For more details on primitive words in $F_2$, see for example~\cite{serInt} or~\cite{CMZ}.
  
 Later it will be essential to distinguish between a word $w_{p/q}$ and its inverse, while for an arbitrary generator pair $(u,v)$ we need to distinguish between $uv$ (or its cyclic conjugate $vu$),  and $uv^{-1}$ (or  its cyclic conjugate $v^{-1}u$). We do this using:
  \begin{definition} \label{signchoice}  
The word $w  \in F_2$ is  \emph{positive} if it is cyclically shortest and if all exponents of $a$ in $w$ are positive. A generator pair $(u,v)$ is positive if both $u$ and $v$ are positive. 
  \end{definition}
  
We remark that if $(u,v)$ is positive  then $||uv||_{a,b} = ||u||_{a,b} + || v||_{a,b}$.
In particular,  the standard  word $w_{p/q}$ constructed as indicated in Figure~\ref{fig:farey} is positive,  as is the standard generator pair $(w_{p/q},w_{r/s})$ whenever $p/q, r/s \in \Qhat$ are neighbours, see also Figure~\ref{fig:regionboundary}.

 \subsection{Fibonacci growth}

Since all words in an extended conjugacy class have the same length, and since $w_{p/q}$  can found by concatenation starting from the initial generators $(a, b)$, it follows that $||w||_{(a,b)}= p+q$  for all $w \in [p/q]$.  
 This leads to the following definition from~\cite{bow_mar}:

 \begin{definition} \label{fibonaccidefn}   A representation $\rho \co F_2 \to \SL $ has \emph{Fibonacci growth}  if there exists  $c>0$ such that for all cyclically reduced words $w \in \P$ we have $  \log^+|\Tr \rho(w)| < c ||w||_{(a,b)}$ and $  \log^+|\Tr \rho(w)| >  ||w||_{(a,b)}/c$
 for all but finitely many cyclically reduced $w \in \P$ where $\log^+x = \max \{0, \log |x| \}$. \end{definition}
Notice that although the definition is made relative to a fixed pair of generators for $F_2$, it is in fact independent of this choice.

The following  result is fundamental. It is proved using the technology described in the next section.
\begin{proposition}[\cite{bow_mar} Proof of Theorem 2, \cite{tan_gen} Theorem 3.3] \label{fibonacci} If $\rho \co F_2 \to \SL$ satisfies the $BQ$-conditions  then $\rho$ has Fibonacci growth. \end{proposition}

\section{More on the Bowditch condition}\label{Bowditchbackground}

In this section we explain some further background to the $BQ$-conditions. For more detail see~\cite{bow_mar} and~\cite{tan_gen}, and for a quick summary~\cite{sty}. The \emph{Farey tree} $\T$ is the trivalent dual tree to the tessellation $\F$, shown superimposed on the left in Figure~\ref{fig:farey}.   As above, $\overline \P$ is identified $\Qhat$ and hence with the set $\Omega$ of complementary regions of   $\T$.  We label the region associated to a generator $u$ by $\bf u$, thus $\bf u' = \bf u$ for all $u' \sim u$. If $e$ is an edge of 
$\T$ we denote the adjacent regions by $\bu(e), \bv(e)$.

 For a given representation $\rho \co F_2 \to \SL$,  note that $\Tr [U,V]$ and hence $\mu =  \Tr [A,B]+2$ is independent of the choice of generators of $F_2$, where as usual $U = \Tr \rho(u)$ and so on. 
 Since $  \Tr U$ is constant on extended equivalence classes of generators, for $\bu \in \Omega$ we can define $\phi(\bu) =  \phi_{\rho}(\bu)=  \Tr U$ for any $ u \in \bu$.  
 For notational convenience we will sometimes write $\hat \bu$   
 in place of $\phi(\bu)$. 
 
 For matrices $X,Y \in \SL$  set $x = \Tr X, y = \Tr Y, z = \Tr XY$. Recall the trace relations:
\begin{equation}\label{eqn:inverse} \Tr XY^{-1} = xy-z \end{equation} and 
\begin{equation} \label{eqn:commreln}  x^2+y^2+z^2 = xyz + \Tr {[X,Y]} +2.
\end{equation}
Setting $\mu =  \Tr {[X,Y]} + 2$, this last equation takes the form 
$$x^2+y^2+z^2 - xyz = \mu.$$

As is well known and can be proven by applying the above trace relations inductively, if  $\bu,\bv,\bw$ is a triple of regions round a vertex  of $\T$, then $\hat \bu,\hat \bv,\hat \bw$ satisfy~\eqref{eqn:commreln} with $x = \hat \bu$ and so on. Likewise if $e$ is an edge of 
$\T$ with adjacent regions $\bu,\bv$ and if $\bw, \bz$ are the third regions at either end of $e$, then $\hat \bu,\hat \bv,\hat \bw, \hat \bz$ satisfy~\eqref{eqn:inverse}, that is, $\hat \bz = \hat \bu \hat \bv-\hat \bw$.  (A map  $\phi: \Omega \to \CC$ with this property is called a \emph{Markoff map}  in~\cite{bow_mar}.)

Given $\rho \co F_2 \to \SL$, let $e$ be an edge of $ \T$  and suppose that the regions meeting its two end vertices are $ \bw,  \bz$.
Following Bowditch~\cite{bow_mar}, orient $e$ by putting an arrow from $ \bz$ to $\bw$ whenever $|\hat \bz| > |\hat \bw|$. If both moduli are equal,  make either choice; if the inequality is strict, say that the edge is \emph{oriented decisively}.  We denote the oriented edge by $\vec e$ and  refer to this oriented tree as the \emph{Bowditch tree}, denoted $ \T_{\rho}$.  If $\vec e$  is a directed edge then its \emph{head} and \emph{tail} are its two ends, chosen so that the arrow on $\vec e$ points towards its head.

We say a path of oriented edges $\vec e_r, 1 \leq r \leq m$  is \emph{descending  to}  $\vec e_m$  if  the head of $\vec e_r$ is the tail of $\vec e_{r+1}$ for  $r = 1, \ldots, m-1$. It is \emph{strictly descending} if each arrow is oriented decisively.  A vertex at which all three arrows are incoming is called a \emph{sink}.

For any $m \geq 0$ and $\rho \co F_2 \to \SL$ define $\Omega_{\rho}(m) = \{ \bu \in \Omega : |\phi_{\rho}(\bu)| \leq m\}$.  From the definition, if $\rho \in \B$ then $\Omega_{\rho}(2)$ is finite and $\phi(\bu) \notin [-2,2]$ for $\bu \in \Omega$.

These following two lemmas show that starting from any directed edge $\vec e_1$, there is a unique descending path to an edge $\vec e_{m}$ which is adjacent to a region in $\Omega(2)$.

\begin{lemma}[{\cite [Lemma 3.7]{tan_gen}}]\label{forkvertex} 
Suppose $\bu,\bv,\bw \in \Omega$ meet at a vertex $\q$  of $\T_{\rho}$ with the arrows on both the edges adjacent to $\bu$ pointing away from $\q$. Then either $|\phi(\bu)| \leq 2$ or $\phi(\bv) = \phi(\bw) = 0$.  In particular, if $\rho \in \B$ then $|\phi(\bu)| \leq 2$. \end{lemma}

\begin{lemma}[{\cite[Lemma 3.11]{tan_gen} and following comment}] \label{infiniteray}
Suppose $\beta$ is an infinite ray consisting of a sequence of edges of $\T_{\rho}$ all of whose arrows point away from the initial vertex. Then $\beta$ meets at least one region $\bu \in \Omega$  with $|\phi( \bu)| < 2$. \end{lemma}

\begin{lemma}\label{connected} 
For any $m \ge 2$, the set  $\Omega_{\rho}(m)$  is connected. Moreover if $\rho \in \B$ then $|\Omega_{\rho}(m)|<\infty$. 
\end{lemma}
\begin{proof} The first statement is~\cite{tan_gen} Theorem 3.1(2). 
That $\Omega_{\rho}(m)$ is finite follows from  Proposition~\ref{fibonacci}, see~\cite{tan_gen} P. 773. \end{proof}

The result which we mainly use is the following:

\begin{theorem}\label{sinktree} There is a constant   $M_0 \geq 2$ and  a finite connected non-empty subtree tree $T_F$  of $\T_{\rho}$ so that for every edge $\vec e$ not in $T_F$,  there is a strictly descending path   from $\vec e$  to an edge of $T_F$.
Moreover  if regions $\bu,\bv$ are adjacent to an edge of $\T$, then $  |\Tr U|, |\Tr V| \leq M_0$   implies $e \in T_F$.  For any $M\geq M_0$, the tree $T_F = T_F(M_0)$ can be enlarged  
to  a larger tree $T_F(M)$ with similar properties, and in addition $T_F$ can be enlarged to include any finite set of edges.  \end{theorem}
\begin{proof} Most of the assertions are proved  on p. 782 of ~\cite{tan_gen}, see also Corollary 3.12 of~\cite{bow_mar}.  To see that $T_F$ can always be enlarged to a tree $T_F(M)$ with similar properties, see the proofs of Theorem 3.2 of~\cite{tan_gen} and Theorem 3.16 of~\cite{bow_mar}.  (In fact there is a precise condition to determine which edges are in $T_F$, see~\cite{tan_gen} Lemma 3.23.) Finally, let $\mathcal K$ be any finite subset of $\T$ and let 
$M = \max \{ \phi(\bu), \phi(\bv)  : \bu, \bv \ \ \mbox{\rm{are adjacent to an  edge in}} \ \  \mathcal K \}$. Enlarging $T_F$ to $T_F(M)$ the result is clear.
\end{proof}

\begin{definition} \label{wake} Let  $\vec e$  be a directed edge.  The  \emph{wake} of $\vec e$, denoted $\W(\vec e)$, is the set of regions whose boundaries are contained in the component of $\T \setminus \{\vec e\}$ which contains the tail of $\vec e$, together with the two regions adjacent to $\vec e$.  \end{definition}
We remark that the wake $\W(\vec e)$ is the subset of $\Omega$ denoted $\Omega^{0-}(\vec e)$ in~\cite{bow_mar} and \cite{tan_gen}. 
Also denote by $  \W_{\E}(\vec e)$ the set of edges $\vec e$ which are adjacent to two regions in $\W(\vec e)$.

Theorem~\ref{sinktree} says that if $\vec e \notin T_F$ then the arrow on $\vec e$ points towards $T_F$. We note the following  slight variation:

 \begin{lemma}\label{wake1} If $\vec e  \notin  T_F$  then every edge in $\W_{\E}(\vec e)$  is oriented towards $\vec e$. \end{lemma}
  \begin{proof}    This follows easily from the definitions. In detail,  let $\dd(T_F)$ be the boundary of $T_F$,  that is, the set of edges in $T_F$ whose tails meet the head of an edge not in $T_F$. If $\vec e  \in \dd(T_F)$  then  by Theorem~\ref{sinktree} the arrow on every edge in $\W_{\E}(\vec e)$ points towards $\vec e$.  Now suppose that $\vec e \notin \dd(T_F)$ and that $\vec f \in \W_{\E}(\vec e)$. Suppose that the descending path $\b(e) $ from $ \vec e$ lands on $\vec g \in \dd(T)$  while the descending path $\b(f) $ from $ \vec f$ lands on $\vec h \in \dd(T)$. Then $\b(e)  \subset \W_{\E}(\vec g)$ while $\vec f \in \b(f)  \subset \W_{\E}(\vec h)$. Since $\W_{\E}(\vec g)$ and $\W_{\E}(\vec h)$ are disjoint  unless $g = h$  and $\vec f \in \W_{\E}(\vec e) \subset  \W_{\E}(\vec g)$ this gives the result.
  \end{proof}

Finally, for the proof of Theorem~\ref{direct} we need the following refinement of Theorem~\ref{fibonacci}, which is a minor variation of Lemmas 3.17 and Lemma 3.19 of~\cite{tan_gen}. 
For $\bu \in \W(\vec e)$ let $d(\bu)$ be the number of edges in the shortest path from $\bu$ to the head of $\vec e$.
Following~\cite{tan_gen} P.777,  define the \emph{Fibonacci function} $F_{\vec e}$ on $\W(\vec e)$ as follows:
$F_{\vec e}(\bw) = 1$ if $\bw$  is adjacent to $\vec e$ and 
$F_{\vec e}(\bu) = F_{\vec e}(\bv)+ F_{\vec e}(\bw)$  otherwise,  where  $\bv,\bw$ are the two regions meeting $\bu$ and closer to $\vec e$ than $\bu$, that is, with $d(\bv) < d(\bu),d(\bw) < d(\bu)$.

\begin{lemma} \label{Increasing3} \label{fibonacciwake} Suppose that $\rho \in \B$
and that $\Vec e$ is a directed  edge such at most one of the adjacent regions is in $\Omega(2)$.  Suppose also that no edge in 
$\W_{\E}(\vec e)$ is adjacent to regions in $\Omega(2)$ on both sides. 
Then there exist $c>0, n_0 \in \NN$, independent of $\Vec e$ (but depending on $\rho$),  so that $\log |\phi_{\rho}(\bu)| \geq c F_{\vec e}(\bu)$ for all but at most $n_0$ regions  $\bu \in \W(\vec e)$. \end{lemma}
\begin{proof} This essentially Lemmas 3.17 and 3.19 of~\cite{tan_gen}, see also Corollary 3.6 of~\cite{bow_mar}.

Since $\Omega(M)$ is finite for any $M>2$, the set  $\{ \log |\phi(\bu)|: \bu \notin \Omega(2) \} $ has a minimum $m > \log 2$. By Lemma 3.17,  if neither adjacent region to $\Vec e$ is in $\Omega(2)$, we can  take $c = m - \log 2$ and $n_0 = 0$. 

Suppose then that exactly one of the adjacent regions $\bx_0$ to $\Vec e$ is in $\Omega(2)$. To apply Lemma 3.19, we need to verify that $ \W(\vec e) \cap \Omega(2) = \{\bx_0\}$.  Note that no region which meets the boundary   $\partial \bx_0$ of $\bx_0$ can be in $\Omega(2)$ by hypothesis.  Let $\vec \e_n, n \in \NN$ be the oriented edges whose heads meet  $\partial  \bx_0$ but which are not contained in $\partial \bx_0$, numbered so that $\vec \epsilon_1$ is the edge not contained in $\partial \bx_0$
whose head meets $\vec e$.
Then  neither of the two adjacent regions to $\vec \epsilon_n$ are   in $\Omega(2)$ for any $n$. It follows from Lemma 3.17 that $\W(\vec \epsilon_n) \cap \Omega(2) = \emptyset$ for $n \in \NN$. Since clearly 
$ \W(\vec e) =  \{\bx_0\} \cup \bigcup_{n\in \NN} \W(\vec \epsilon_n)$ the claim follows.  

Now Lemma  3.19 gives $c>0$ and  $n_0\in \NN$, depending only on $\bx_0$, so that  $\log |\phi_{\rho}(\bu)| \geq c F_{\vec e}(\bu)$ for all but at most $n_0$ regions  $\bu \in \W(\vec e)$. Since  $\Omega(2)$ is finite and 
$\bx_0 \in \Omega(2)$, we can adjust the constants so as to be uniform independent of $\vec e$. \end{proof}

\subsection{The W-arrows}\label{Warrows}
 There is another way to orient the edges of $\T$, this time in relation to word length. For $\bu \in \Omega$, define $||\bu|| = ||u||_{(a,b)}$ for any  cyclically reduced positive word  $u  \in \bu$; clearly this is independent of the choice of $u$.  Provided $e$ is not the edge $e_0$ separating the regions $(\bf a,\bf b)$, then if $\bz,\bw$ are the regions at the two ends of $e \in \T$, put an arrow pointing from $\bz$ to $\bw$ whenever $||\bz||_{a,b} > ||\bw||_{a,b}$. 
We call these arrows,  $W$-arrows, while the previously assigned arrows defined by the condition $|\phi(\bz)| \geq |\phi(\bw)|$ we refer to as $T$-arrows (for word length and trace respectively). Clearly every edge is connected by a strictly descending path of $W$-arrows to one of the two vertices at the ends of   the edge $e_0$. \emph{We retain the notation $\vec e$ exclusively to refer to the orientation of the $T$-arrow, likewise the terms head and tail.} 

 If $e$ is an edge of $\T$, as usual   denote by $\bu(e), \bv(e)$ the regions adjacent to $e$.  Notice that  if $u\in  \bu(e), v \in \bv(e)$ are a positive generator pair, then we have  $||\bf{uv}||> ||\bf{uv^{-1}}||$ so that the $W$-arrow points from $\bf{uv}$ to $\bf{uv^{-1}}$.

  For $N \in \NN$ 
let $B((a,b), N) = \{ e \in \T  : \max \{||\bu(e)||_{a,b}, ||\bv(e)||_{a,b}\} \leq N \}$. 
The next proposition shows that for all but finitely many arrows, the $W$- and $T$- arrows point in the same direction.

  \begin{proposition} \label{wordsandtraces} There exists $N_0>0$ such that if $\vec e \notin B((a,b), N_0)$ is an oriented edge of $\T_{\rho}$ with regions $\bz,\bw$ at its tail and head respectively, then $||\bz|| > ||\bw||$.
 \end{proposition}
 \begin{proof} This is a general result about  attracting trees. Enlarge the finite sink tree $T_F$  of Theorem~\ref{sinktree} if necessary so that  $e_0 \in \T_F$. Choose $N_0$ large enough that $T_F(M_0) \subset B= B((a,b), N_0)$.   Then every edge not in  $B$  is  connected by a path of decreasing $T$-arrows to an edge of $T_F$.

 If the result is false, there is an edge $\vec e$ not in $B$ with regions $\bz,\bw$ at its tail and head respectively such that  $||z||_{a,b} < ||w||_{a,b}$ for $z \in \bz, w \in \bw$.    By Lemma~\ref{wake1}, every edge in $\W_{\E}(\vec e)$ is connected by a strictly descending path of $T$-arrows to the tail of $\vec e$.
On the other hand, $\vec e$ is connected by a strictly descending path of $W$-arrows to one of the two vertices at the ends of 
$e_0$. But these $W$-arrows are contained in $\W(\vec e)$ and, following on from the initial edge $e$, must  all point in the opposite direction to the $T$-arrows. Thus
one of the two vertices at the ends of 
$e_0$ is outside $B$, which is impossible.
\end{proof}
 
  \begin{corollary} \label{headsandtails} There exists $N_0 \in \NN$ such that if $\vec e$ is an edge outside $B(N_0)$, then  every edge $\vec f \in \W(\vec e)$  has head  $\bf{uv}^{-1} $ and  tail   $\bf {uv}$ whenever $u \in \bu(f), v \in \bv(f)$ are  a positive generator pair associated to $\vec f$.
 \end{corollary}

 \section{Results from~\cite{LX}}\label{geometry} 

In this section we  collect  the main  results from~\cite{LX} needed to prove Theorem~\ref{BQimpliesPS}.
 \subsection{The double cone lemma}\label{cone} 
 
 Suppose that $H, H'$ are hyperbolic half planes and let $\hat H$ be one of the two closed half spaces defined by $H$. By an inward (resp. outward) pointing normal to $\hat H$ we mean a normal to $H$ which points into (resp. out of) $\hat H$. If $\hat H'$ is another half space  such that $\hat H\supset  \hat H'$ and   $d(   H,   H')>0$  we say that $\hat H, \hat H'$ are \emph{properly nested}.

 \begin{lemma}\label{conelemma}   Suppose $ 0 < \a < \pi/2$. Then there exists $L_0>0$ with the following property. 
Suppose that   $H, H'$ are hyperbolic half planes defining half spaces $\hat H, \hat H'$. Let $\M$ be a line joining points $O \in H, P \in H'$ such that $\M$ is orthogonal to $\hat H'$ and makes an  angle $0 \leq \theta < \a$ with the inward pointing normal to $\hat H$.  Then  $\hat H \supset  \hat H'$ are properly nested whenever $d(O,P)> L_0$. \end{lemma}
  \begin{proof}  
  
  If this is false, then $H'$ meets $H$ in a point $Q \in \HH^3 \cup \dd \HH^3$. Then $OPQ$ is a triangle with angle $\psi = \pi/2 - \theta$ at $O$ and $\pi/2$ at $P$. Let $ L_0$ be the length of the finite side of a triangle with angles $\pi/2 - \a, \pi/2,0$. Since 
  $\psi = \pi/2 - \theta > \pi/2 - \a$ then  $d(O,P)< L_0$. Clearly from the directions of the normals, $\hat H\supset  \hat H'$ and moreover  $d(  H,    H')>0$.
     \end{proof} 
  
   \begin{corollary}\label{doubleconelemma}(\cite{LX} Lemma 3.5) 
   Suppose that $H, H'$ are hyperbolic half planes with corresponding half spaces $\hat H, \hat H'$ and let  $\M$ be a line joining points $O \in H, P \in H'$ which makes 
  angles $0 \leq \theta, \theta' \leq \a$ with the inward pointing normal to $\hat H$ and the outward pointing normal to $\hat H'$ respectively.  Then $\hat H\supset  \hat H'$ are properly nested provided  $d(O,P)> 2L_0$.
    \end{corollary}
  \begin{proof}  Let $H''$ be the plane perpendicular to $\M$ through its mid-point and apply Lemma~\ref{conelemma} to 
  $H, H''$ and $H'', H'$.
    \end{proof}

  \subsection{Generators and the amplitudes of a right angled hexagon}\label{sec:amplitude}

Let $\H$ be a right angled hexagon  with  consistently oriented sides $s_1, \ldots , s_6$ and let $\s_i$ be the complex distance between  sides $s_{i-1}, s_{i+1}$. 
  The amplitude $\Am(\s_{i-2}, \s_i, \s_{i+2})$  introduced in~\cite{Fen} VI.5, is, up to sign, an invariant  of the  triple of alternate sides $s_{i-2}, s_i, s_{i+2}$. Its importance is that  if  $\H$ is constructed as described below from a positive ordered generator pair $(u,v)$, then up to sign the amplitude relative to the three sides $\Ax U, \Ax V, \Ax U^{-1}V^{-1}$,  is the trace of the  square root of commutator ${U,V}$ and hence  independent of the choice of generators. This point was used crucially in~\cite{LX}.

  \begin{definition} \label{amplitude}  Let $\H$ be a consistently oriented right angled hexagon  with oriented sides $ s_1, \ldots ,  s_6$  and let  $\s_i$ be the complex distance  between sides $s_{i-1}, s_{i+1}$. Define the \emph{amplitude}   $\Am(\s_1,\s_3 , \s_5) = -i \sh \s_2\sh  \s_3 \sh \s_4$.
 \end{definition}
See for example~\cite{Fen} or \cite{serwolp} for a discussion of complex length and hyperbolic right-angled hexagons.

 Let $\s_{14}$ be the complex distance between the oriented lines $s_1$ and $s_4$.  Using the cosine formula in the oriented right angled pentagon with the sides $s_1, s_2, s_3, s_4, s_{14}$ (where  $s_{14}$  is the common perpendicular of $s_1$ and $s_4$, oriented from $s_1$ to $s_4$), we find  $ \ch \s_{14} = - \sh \s_2\sh \s_3$. Thus we can alternatively write the amplitude as 
 $Am(\s_1,\s_3 , \s_5) =  i \ch \s_{14} \sh \s_4$.

We now fix a choice of lift   $R \in \SL$ of the order two rotation about an oriented line using line matrices as described in \cite{Fen} V.2.  Denote the  oriented line with endpoints $\zeta, \zeta' \in \Chat$,  oriented from $\zeta$ to $\zeta'$,  by $[\zeta, \zeta']$.
The \emph{line matrix} $  R  ([\zeta, \zeta']) \in \SL$ is a choice of matrix representing the $\pi$-rotation  about $[\zeta, \zeta']$. 
If $\zeta, \zeta' \in \CC$ then  
$$  R  ([\zeta, \zeta']) =   \dfrac{i}{\zeta'-\zeta} \begin{pmatrix} \zeta + \zeta'& -2 \zeta \zeta' \cr  2 & -\zeta - \zeta'\end{pmatrix},$$
while  $$  R  ([\zeta, \infty])=   i \begin{pmatrix} 1& -2 \zeta  \cr  0 & -1\end{pmatrix}, \ \ R([\infty, \zeta' ])  =  - i \begin{pmatrix} 1& -2 \zeta  \cr  0 & -1\end{pmatrix}.$$
As shown~\cite{Fen}, this definition respects the orientation of lines and is invariant under conjugation in $\SL$.

If $R_i$ is the line matrix associated to the oriented side $s_i$ of $\H$ as above, then  $R_i^2 = -id$ and $R_iR_{i+1} = -R_{i+1} R_i$ . Moreover $R_{i-1}R_{i+1}   $ is a loxodromic which translates by complex distance $2  \s_i$  along an axis which extends  $s_i$. By  ~\cite{Fen} V.3,   $\Tr R_{i-1}R_{i+1}    =  - 2 \cosh \s_i$  and 
$\Tr R_{i -1}R_iR_{i+1}    = -2 i \sh \s_i$. These formulae can  be easily checked by letting $\zeta = e^{\s_i}$ and arranging $s_{i-1}, s_i$ and $s_{i+1}$ to be the oriented lines joining $[-1,1] , [0,\infty], [-\zeta, \zeta] $ respectively so that  
$$  R_{i-1} = \begin{pmatrix} 0 & i \cr  i & 0 \end{pmatrix},  \ R_{i } =   \begin{pmatrix} i & 0 \cr  0 & -i \end{pmatrix}, \   R_{i+1}=     \begin{pmatrix} 0 & i\zeta \cr  i/\zeta & 0 \end{pmatrix}. $$  

It follows from the above formulae, that we can alternatively define  $\Am(\s_1,\s_3 , \s_5) =-\dfrac{1}{2}\Tr (R_5  R_3  R_1)$. Moreover this expression is unchanged under even cyclic permutations and changes sign under odd ones.

 We now explain the invariance of the amplitude under change of generator.
 Suppose that $(u,v)$ is a positive ordered generator pair.  
Construct an oriented right angled hexagon $\H = \H(u,v)$ with the axes of $(U,V,  U^{-1}V^{-1})$ oriented in their natural directions, i.e. pointing in their respective translation directions, forming three alternate sides.  The orientations of the three remaining sides then follow.  We call this the \emph{standard hexagon} associated to $(u,v)$.

\begin{proposition}\label{amplitudeinvt}
Let $\H = \H(u,v)$ be the standard hexagon associated to the image of an positive ordered generator pair $(u,v)$. Let $s_2 = \Ax U, s_4 = \Ax V, s_6 = \Ax U^{-1}V^{-1} $ and label the other sides accordingly. Then up to sign, $\Am (\s_1,\s_3 , \s_5)$ is independent of the choice of $(u,v)$.
\end{proposition}  
\begin{proof}
 With $\H = \H(u,v)$ as defined in the statement, we have $R_{3} R_1 = U $ and $R_{5}R_3   = V$ so that $R_1 R_{5} = U^{-1}V^{-1}$. Hence 
$$UVU^{-1}V^{-1} = R_{3} R_1R_{5}R_3 R_1R_3 R_3 R_5 = - ( R_{3} R_1R_{5})^2.$$
On the other hand, 
$$   \Tr (R_{5} R_4R_{3}) \Tr (R_4 R_1) =  \Tr (R_{5} R_4R_{3} R_4 R_1)+ \Tr (R_{5} R_4R_{3} R_1 R_4)=-  2 \Tr R_{5} R_3R_{1}.$$
By the above, $     \Tr (R_{5} R_4R_{3}) \Tr (R_4 R_1)  = -4 i \sh  \s_4 \ch \s_{14}  = 4  \Am(\s_1,\s_3 , \s_5)$. 
Since as we have seen the trace of the commutator is an invariant of generator triples,  
it follows that so is  $\Am^2(\s_1,\s_3 , \s_5)$ and hence, up to sign, so is $\Am(\s_1,\s_3 , \s_5)$.  
 \end{proof}  

We refer to $\Am(\s_1,\s_3 , \s_5) = -i   \sh \delta_{UV}\sh  \lambda(U) \sh \lambda(V)$ as the \emph{amplitude} of $\H(u,v)$.

\subsection{Some simple observations}\label{observations}

We need a few more simple observations.

\begin{lemma} \label{cxinequality} \rm{(See~\cite{bow_mar}.)} Suppose that $\bu,\bv \in \Omega$ are adjacent to an oriented edge $\vec e$ of $\T$ with  $\bw, \bz$ being the regions at the head and  tail of $\vec e$ respectively. Then $ \Re \bigl  ( \dfrac{ \hat  \bz }{\hat \bu \hat \bv} \bigr ) \geq 1/2$, where $\hat \bz = \phi_{\rho}(\bz)$ and so on as in Section~\ref{Bowditchbackground}. \end{lemma}  
   \begin{proof}  
   It is easy to check that if $\xi,\eta \in \CC$ and $\xi+\eta = 1, |\eta| \leq |\xi|$, then $ \Re \xi \geq 1/2$. 
   With $\bu,\bv, \bw, \bz$ as in the statement  
we have $\hat \bz + \hat \bw = \hat \bu  \hat \bv$ and $|\hat \bz |\geq | \hat \bw|$.
Now apply the above  with  $\xi = \dfrac{  \hat \bz}{  \hat \bu \hat \bv} , \eta =  \dfrac{\hat  \bw}{  \hat \bu \hat \bv} $.   
  \end{proof} 

 \begin{lemma} \label{tanhinequality} If $\xi  \in \CC$ and $\Re \xi>0$ then $ \Re (\tanh \xi) \geq 0$. \end{lemma}
   \begin{proof}  If $\xi = x+iy$  then $\Re (\tanh \xi) = \dfrac{\sh x \ch x}{|\ch x \ch y  + i\sh x \sh y|^2 }$.\end{proof} 
   
We will also need a  comparison of hyperbolic translation lengths and traces.  
 
 For a loxodromic element $X \in \SL$ let  $\ell(X)>0$ denote the (real) translation length and let  $\lambda (X) = (\ell(X) + i \theta(X))/2$ be \emph{half} the complex length,  so that $\Tr X = \pm 2 \cosh \lambda(X)$.
  \begin{lemma}\label{compare2} There exists  $L_0>0$ so that if 
 $\xi + i \eta \in \CC$ with $ \xi >L_0 $  then  
$\xi - \log 3 \leq \log |\cosh (\xi + i \eta)|  \leq \xi$. In particular, for $X \in \SL$ we have $e^{\ell(X)} /3\leq |\Tr X|/2 \leq e^{\ell(X)}$ whenever $\ell(X) > L_0$.  
\end{lemma}
\begin{proof}
For the right hand inequality, since $ |\cosh (\xi + i \eta)| = e^{\xi}   | (1+ e^{-2\xi -2i\eta} )|/2$ we have
$$\log |\cosh (\xi + i \eta)| = {\xi} + \log  | (1+ e^{-2\xi -2i\eta} ) |/2 \leq \xi$$  since  $| (1+ e^{-2\xi -2i\eta})  |/2 \leq 1$.

For the left hand inequality, since $\xi > L_0$ we have, choosing $L_0$ large enough, $| (1+ e^{-2\xi -2i\eta})  |/2 \geq 1/3$ so that $\log  | (1+ e^{-2\xi -2i\eta} ) |/2 \geq -\log 3$ and hence 
$\log |\cosh (\xi + i \eta)| \geq \xi -\log 3$.
\end{proof}

\subsection{The key step}\label{keystep} 

We  now come to the  key steps from~\cite{LX}   used to prove Theorem~\ref{BQimpliesPS}. 
 
 \begin{proposition} \label{angleinequality}   (\cite{LX} Lemma 5.1)
Suppose that $\rho \in \B$ and that $ 0 < \a < \pi/2$ is given. Suppose also that as in Lemma~\ref{cxinequality}, $\bu,\bv \in \Omega$ are adjacent to an oriented edge $\vec e$ of $\T$. With $N_0$ as in Corollary~\ref{headsandtails}, suppose $ u \in \bu, v \in \bv$ are a positive generator pair and that $\max \{||u||, ||v||\} > N_0$. 
Let $\delta_{UV}$ be the complex distance between the  axes of $U = \rho(u), V = \rho(v)$, oriented in the direction of positive translation.
 Then there exists $L_1>0$ depending only on $\a$ and $\rho$ such that   $|\Im \delta_{UV}| \leq \a$ whenever $\max \{\ell(U), \ell(V)\} > L_1$.
  \end{proposition}

    \begin{proof}  Without loss of generality, suppose that $ \ell(U) \geq  \ell(V)$. Let $\delta_{UV} = d + i \theta$. Since by assumption $\Omega(2)$ is finite and $\Tr \rho(g) \neq \pm 2$ for all $g \in \P$, there exists $c>0$ such that $|\Tr  \rho(g) \pm 2| > c$ for all $g \in \P$.  Hence $|\sh \lambda(G)|$ is uniformly bounded away from $0$ for all $g \in \P$, where $G = \rho (g)$.  By Proposition~\ref{amplitudeinvt}   
    the absolute value of the amplitude of $\H(u,v)$,  that is, $|   \sh \delta_{UV}\sh  \lambda(U) \sh \lambda(V)|$,  is independent of $(u,v)$.  Combined with Lemma~\ref{compare2}, it follows that provided that $\ell(U)> L_0$ we have
  \begin{equation}\label{distbound}
    | \sh \delta_{UV}| \leq k e^{- \ell(U)} 
  \end{equation}  
  for a  constant $k$ which depends only on the representation $\rho$.   
Since $| \sh \delta_{UV}|^2 = \ch^2 d \sin^2 \theta + \sh^2 d \cos^2 \theta$ we deduce that  $d \to 0$  and either $ \theta \to 0$ or $ \theta \to \pi$ as $\ell(U)\to \infty$.

Now  the cosine formula in $\H(u,v)$ gives
\begin{equation*}
  \cosh \d_{UV}  = \frac{\cosh \lambda{(U^{-1}V^{-1})}  - \cosh \lambda(U)  \cosh \lambda(V)  } {\sinh \lambda(U) \sinh \lambda(V)}
 \end{equation*}
and hence 
\begin{equation*}
  \cosh \d_{UV} \tanh \lambda(U)  \tanh \lambda(V)  = \frac{\cosh \lambda (VU)} { \cosh \lambda(U)  \cosh \lambda(V)  }  - 1
 \end{equation*}
which gives 
\begin{equation*}
1+   \Re  ( \cosh \d_{UV} \tanh \lambda(U)  \tanh \lambda(V)  )  = \Re \Bigl( \frac {\cosh \lambda (VU)} { \cosh \lambda(U)  \cosh \lambda(V)} \Bigr). 
 \end{equation*}

By Corollary~\ref{headsandtails} the $T$-and $W$-arrows on $\vec e$ agree. Hence $\bu \bv$ is the region at the  tail of $\vec e$ and $\bu \bv^{-1}$ the one at its head.  Thus by Lemma~\ref{cxinequality} we have  $ \Re \Bigl( \dfrac {\cosh \lambda (VU)} { \cosh \lambda(U)  \cosh \lambda(V)} \Bigr) \geq 1$ from which it follows that 
 $  \Re  ( \cosh \d_{UV} \tanh \lambda(U)  \tanh \lambda(V)  ) \geq 0$. By Lemma~\ref{tanhinequality}, $ \tanh \lambda(U) , \tanh \lambda(V)   \geq 0$  so that $\Re \cosh \d_{UV}  = \ch d \cos \theta \geq 0$  from which we deduce that $\theta \to 0$.
 This completes the proof.   \end{proof}

 \begin{proposition} \label{sephyperplanes}   (\cite{LX} Theorem 5.4)
Suppose that  $\bu,\bv \in \Omega$ are adjacent to an  edge $e$ of $\T$. Then there is a half space $\hat H$  and $L_2>0$ so that if  $\max \{ \ell(U), \ell(V)\} \geq L_2$, then for any $X,Y \in \{ U,V \}$, the half spaces $ X^{-1} \hat H \supset \hat  H \supset Y\hat  H$ are properly nested.
\end{proposition}
\begin{proof} 

Suppose for definiteness that $\ell(U) \geq \ell(V)$. Let $H$ be the hyperplane orthogonal to $\Ax V$ and containing  the common perpendicular $D$ to $\Ax U, \Ax V$. Let $\hat H$ be the half space cut off  by $H$  and containing the forward pointing unit tangent vector $\bf t_V$ to $\Ax V$ at $P = \Ax V \cap D$.   Note that  $ V^{-1} \hat H \supset \hat  H \supset V \hat  H$ are properly nested since $V$ is loxodromic and translates  $H$  disjointly from itself.   
    
 Now suppose  $Y = U$. Note that for $L$ sufficiently large, by Proposition~\ref{fibonacci}, $\ell(U) > L$ implies that $||u||_{a,b}> N_0$ with $N_0$ as in Proposition~\ref{angleinequality}. Hence
by Proposition~\ref{angleinequality}  we can choose $L= L_1(\pi/4)$ so that $|\Im \d_{UV}| \leq \pi/4$ whenever $\ell(U) \geq L$.   Let $ Q$ be the intersection point  of $  \Ax U$ with $D$ and let  $\bf t_U$ be the forward pointing unit tangent vector along $\Ax U$ at $Q$. Then $\bf t_V$  is translated by distance $\Re \d_{UV}$ and  rotated by angle $\Im \d_{UV}$ along $D$ to coincide with $\bf t_U$ at $Q$. Thus $\bf t_U$
 makes an angle  at most $\pi/4$ with the inward pointing normal $\bf n_Q$ to $\hat H$ at $Q$.  Likewise $U(\bf t_U)$ makes an angle  at most $\pi/4$ with the inward pointing normal $U(\bf n_Q)$ to $U(\hat  H)$.
 It follows by Corollary~\ref{doubleconelemma} that for  $\ell(U)$  sufficiently large, the half planes 
 $\hat H \supset U (\hat H)$  are properly nested and hence so are $U^{-1} (\hat H) \supset   \hat H$.  
This completes the proof.
\end{proof}

 \begin{proposition}\label{longwordsqgeod}  (\cite[Theorem 5.4]{LX})
 Suppose  that $(u,v)$ is a  positive generator pair such that that  $\max \{ \ell(U), \ell(V)\} >  L_2$ with $L_2 $ as in Proposition~\ref{sephyperplanes}.  
 Let $\C (u,v)$ denote the set of all cyclically shortest words which 
 are products of positive powers of $u$'s and $v$'s. 
 Then the collection of broken geodesics $\{\br_{\rho}(w; (u,v)), w \in \C  (u,v)\}$ is uniformly quasigeodesic. 
\end{proposition}
\begin{proof}   With the notation of Proposition~\ref{sephyperplanes}, pick a basepoint $O$ in the hyperplane $H$
 and let $d$ be the minimum distance between any pair of the planes $H, U(H), V(H)$.
Label the vertices of $\br_{\rho}(w; (u,v))$ in order as $P_n, n \in \ZZ$ with $O = P_0$ and denote the image of $H$  containing $P_n$ by $H_n$.  

Any three successive vertices $P_n, P_{n+1}, P_{n+2}$ are of the form $ZX^{-1}{O}, ZO, ZYO$ for some $X,Y \in \{ U = \rho(u),V= \rho(v) \}, Z \in \rho(F_2)$. Therefore by
Proposition~\ref{sephyperplanes}  the corresponding half spaces $\hat H_n, \hat H_{n+1}, \hat H_{n+2}$ are properly nested.
It follows that each consecutive pair of half spaces in the sequence  $\ldots, \hat H_n, \hat H_{n+1}, \hat H_{n+2}, \ldots$ are properly nested
and hence that $d(P_n, P_m) \geq d(\hat H_n, \hat H_m) = |n-m|d$ which proves the result.
\end{proof}

 \section{The Bowditch condition implies primitive stable}\label{sec:BQimpliesBIP} 
 
In this section we prove Theorem~\ref{introthmA}, that a representation $\rho \co F_2 \to \SL$ satisfies the $BQ$-conditions if and only if  $ \rho$ is primitive stable.

The result in one direction is not hard, see for example~\cite{lupi}. 
 \begin{proposition}\label{PSimpliesBQ}  The condition $PS$  implies the Bowditch $BQ$-conditions.
\end{proposition}
  \begin{proof} Let $u \in \P$. If the broken geodesic $\br(u; (a,b))$ is quasigeodesic then it is neither elliptic nor parabolic, so the first condition $\Tr U \notin [-2,2]$ holds.

   If the collection of broken geodesics $\br(u; (a,b)),  u \in \P$ is  uniformly quasigeodesic then $\br(u; (a,b))$ is at a uniformly bounded distance from $\Ax U$ for each $u \in \P$.  We deduce that
  $$c' ||u||_{a,b} - \epsilon \leq d_{\HH}(O, UO) \leq c + \ell(U) $$  for uniform constants $c,c', \epsilon>0$.
Since only finitely many words have word length less than a given bound,   this implies that only finitely many elements have hyperbolic translation lengths and therefore, by Lemma~\ref{compare2},    traces, less than a give bound. \end{proof}

It remains to  prove the converse.
The following  lemma is well known.
\begin{lemma} \label{singleqgeod} 
Let $w$ be a cyclically shortest word in $F_2$ and let $\rho \co F_2 \to \SL$. Suppose that the image   $W = \rho(w)$  is loxodromic  and that  $(u,v)$ is a generator pair. Then the broken geodesic $\br_{\rho}(w; (u,v))$ is quasigeodesic with constants depending only on $\rho, w, $ and $( u, v)$.
\end{lemma}
\begin{proof} 
Suppose that $||w||_{(u,v)} = k$ and number the vertices $P = \rho(x)O, x \in F_2$ of $\br_{\rho}(w; (u,v))$ in order as $P_r, r \in \ZZ$ with $P_0 = O$. We have to show that there exist constants $K,\e>0$ so that  if  $n<m$ then
$$(m-n )/K - \e\leq  d(P_n,P_m) \leq K(m-n) + \e.$$

Pick $c>0$ so  that $  d(O,\rho(h)O) \leq c $ for  $h \in \{ u,v\}$. Clearly  $d(P_n,P_m) \leq c (m-n)$.
For the lower bound, write  $m -n = rk + k_1$ for $r \geq 0, 0 \leq k_1 < k$.
Then  for some cyclic permutation of $w$, say $w'$, setting  $W' = \rho(w')$ we have
$W'^r(P_n) = P_{n+rk} $ so that  $d(P_n, P_{n+rk} ) \geq r \ell(W)$. 
Thus $$    d(P_n,P_m) \geq  d(P_n,P_{n+rk} ) - d(P_{n+rk}, P_m) \geq   (m-n) \ell(W)/k -       kc - \ell(W)/k.$$
\end{proof}

 \begin{theorem}\label{BQimpliesPS}  The Bowditch $BQ$-conditions implies $PS$.
\end{theorem}
 \begin{proof}
Choose a finite sink tree $T_F = T_F(M_0)$ as in Theorem~\ref{sinktree}.
Use Proposition~\ref{wordsandtraces} to enlarge $T_F = T_F(M_0)$ if necessary so that the $W$- and $T$-arrows coincide for every edge outside $T_F$. 
By further increasing $M_0$ if necessary we can assume that
 $|\Tr \rho(u)|> M_0$  implies $\ell(U)> \max \{ L_0, L_2\}$  with $L_0, L_2$ as in Lemma~\ref{compare2} and Proposition~\ref{longwordsqgeod} respectively.

 Suppose now that   $e \notin T_F$. Then at least one of the regions $\bu$ adjacent to $e$ has $\ell(U)> \max \{ L_0, L_2\}$ and moreover 
the $W$- and $T$-arrows on $e$ coincide. Let $\bv$ be the other region adjacent to $e$ and suppose that $u \in \bu, v \in \bv$ are a positive pair, so that  $||uv|| > ||uv^{-1}||$. Since the $W$-arrow on $e$ points the same direction as the $T$-arrow
it follows that   $|\Tr UV| \geq |\Tr UV^{-1}|$.   
For the same reason, every region in $\W(\vec e)$ corresponds to a word which is a product of positive powers of $u$'s and $v$'s. 
Thus by 
Proposition~\ref{longwordsqgeod}  the collection of all broken geodesics corresponding to regions in $\W(\vec e)$ is uniformly quasigeodesic.
 
 Since $T_F$ is finite, there are finitely many edges $\{\vec e_i,i=1, \ldots, k\}$ whose heads meet  $T_F$. Moreover every region not adjacent to an  edge in $T_F$ is in $\W(\vec e_i)$ for some $i $.

There are only finitely many regions $\bw$  adjacent to some edge of $T_F$.
By Lemma~\ref{singleqgeod},  for each  such     $\bw$ and  $w \in \bw$, the broken geodesic $\br_{\rho}(w; (a,b))$ is quasigeodesic with constants depending on $\bw$.
 
 It follows that  there is a finite set of generator pairs $\S$, such that any   $w \in F_2$ can be expressed as a word in some $(s,s' ) \in \S$  in such a way that   $\br_{\rho}(w; (s,s' ))$  is quasigeodesic with constants depending only on $(s,s' )$. For fixed $(s,s')$ each quasigeodesic  $\br_{\rho}(w; (s,s' ))$ can be replaced by a broken geodesic $\br_{\rho}(w; (a,b ))$
which is also quasigeodesic  with a change of constants depending only on $(s,s')$ and not on $w$. The total number of  replacements required involves only finitely many constants and the result follows.
\end{proof}

\section{Palindromicity and the  Bounded Intersection Property}\label{BIP}

It is  easy to prove Theorem~\ref{introthmB}, that $\rho \in \B$ implies that $\rho $ has the bounded intersection property,   using  Theorem~\ref{BQimpliesPS}.
 \begin{proposition} \label{prop:PSImpliesbounded} If a representation $\rho \co F_2 \to \SL$   is primitive stable then it satisfies $BIP$. 
 \end{proposition}
   \begin{proof} The broken geodesic corresponding to any primitive element by definition passes through the basepoint $O$.
The broken geodesics $\{\br_{\rho}(u ; (a,b)) \}, u \in \P$ are by definition uniformly quasigeodesic, so each is  at uniformly bounded distance to its corresponding axis. Hence all the axes are at uniformly bounded distance to $O$ and so in particular axes corresponding to primitive palindromic elements  cut the three corresponding special hyperelliptic axes  in  bounded intervals.\end{proof}

This result is of course much more interesting once we know that all primitive elements have palindromic representatives. We make a precise statement in Proposition~\ref{uniquepalindromes}.
In Theorem~\ref{direct} we then give a  direct proof that $\rho \in \B$ implies that $\rho $ has the bounded intersection property.

\subsection{Generators and palindromicity}\label{genpalin}
Let $\EE = \{0/1,1/0, 1/1\}$ and define a map $\bb \co \hat \QQ \to \EE$  by $\psi(p/q) = \bar p /\bar q$, where   $\bar p ,\bar q$ are the  mod 2 representatives of $p,q$ in $\{0,1\}$. We refer to $\psi(p/q) $ as the mod 2  equivalence class of $p/q$. 
Say $p/q \in \Qhat $ is of type $\w \in \EE$ if  $\psi(p/q) = \w$. Say a generator $u \in F_2$ is of type $\w$ if $u \in [p/q]$ and $p/q$ is of type $\w$; likewise a generator pair $(u,v)$  is type $(\w, \w') $  if  $u, v$ are of types $\w, \w'$ respectively. 
As in Section~\ref{sec:BIP}, we fix once and for all a generator pair $(a,b)$ and identify $a$ with $0/1$, $b$ with $1/0 $ and $ab$ with $1/1$. The \emph{basic generator pairs} are the three (unordered) generator pairs $(a,b)$, $(a,ab)$ and $(b,ab)$ corresponding to $(0/1,1/0 )$, $(0/1,1/1 )$ and $(1/0,1/1 )$ respectively.  (Here the order $ba$ or $ab$ is not important but fixed.) 
For $\w,\w' \in \EE$ we say $u$ is palindromic with respect to $(\w,\w'),  \w \neq \w'$  if it is palindromic when rewritten in terms of the basic pair of generators corresponding to $(\w,\w')$; equally we say that a generator pair  $(u,v)$  is cyclically shortest (respectively palindromic with respect to the pair  $(\w,\w')$) if  each of $u,v$  have the same property. We refer to a generator pair  $(u,v)$ which is palindromic with respect to some pair of generators, as a \emph{palindromic pair}. Finally, say a generator pair $(u,v)$ is conjugate to a pair $(u',v')$ if there exists $g  \in F_2$ such that 
 $gug^{-1} = u'$ and $gvg^{-1} = v'$.

 \begin{proposition} \label{uniquepalindromes}
  If $u \in \P$  is positive and of type $\w \in \EE$, then,  for each $\w' \neq \w$,  there is exactly one conjugate generator $u'$ which  is positive and palindromic with respect to $(\w,\w')$. 
If $(u,v)$ is a  positive generator pair of type $(\w,\w')$, then  there is exactly one conjugate generator pair $(u',v')$ which  is positive and palindromic with respect to $(\w,\w')$.  
  \end{proposition} 
\begin{proof} 
We begin by proving  the existence part of the second statement.  Observe that the edges of the Farey tree $\T$ may be divided into three classes, depending on the mod two equivalence classes of the generators labelling the neighbouring regions. 
In this way we may assign colours $r,g,b$ to the pairs 
$(0/1,1/0); (0/1, 1/1 ); (1/0, 1/1 )$ respectively and extend to a map $col$ from edges to $\{r,g,b\}$, see Figure~\ref{fig:colouredtree}. Note that no two edges of the same colour are adjacent, and that the colours round the boundary of each complementary region alternate.

\begin{figure}[ht]
\includegraphics[width=7.5cm]{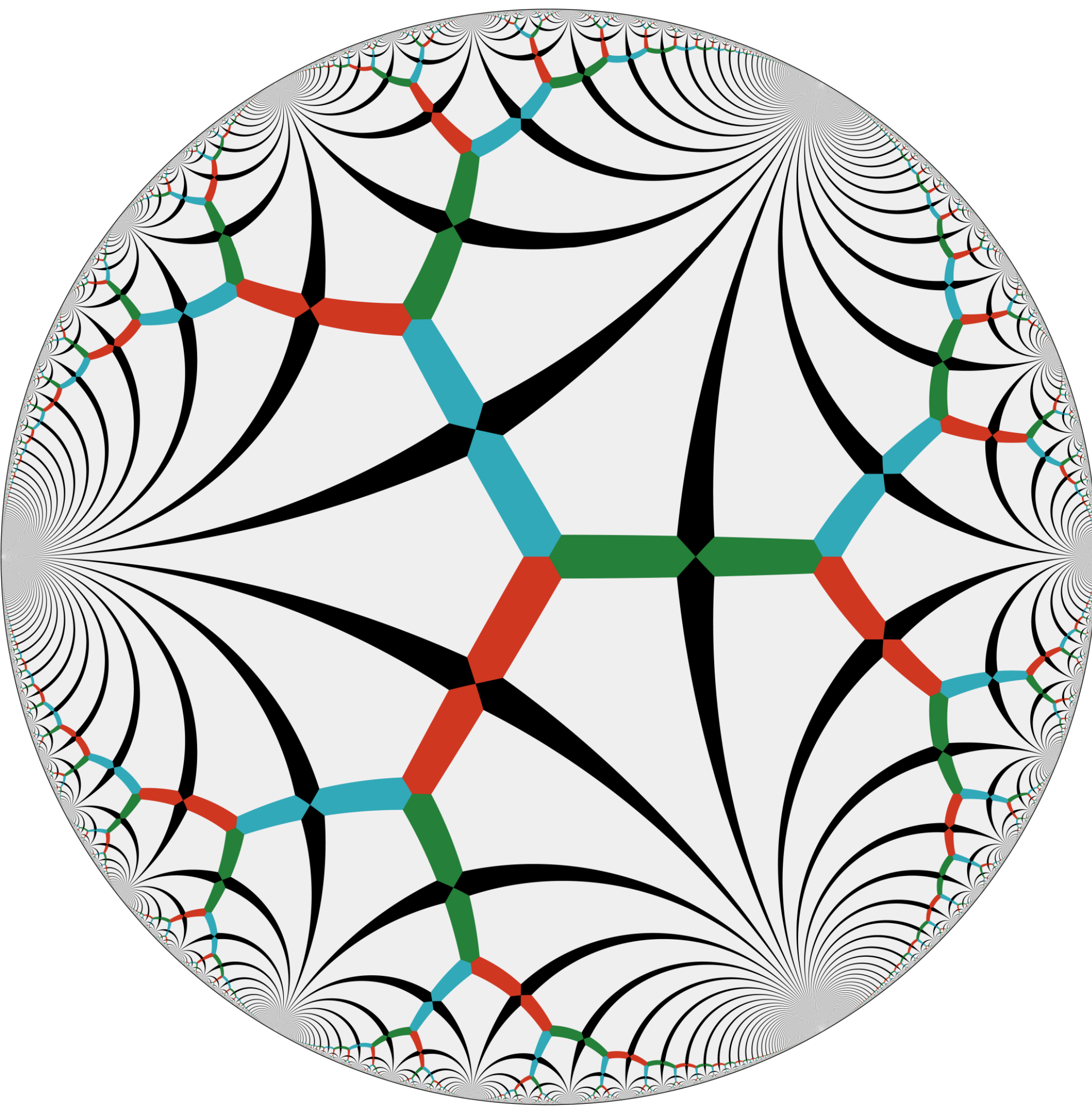}
\caption{The coloured Farey tree. The colours round the boundary of each complementary region alternate. The picture is a conjugated version of the one in Figure~\ref{fig:farey}, arranged so as to highlight the three-fold symmetry between $(a,b,ab)$. Image courtesy of Roice Nelson.}\label{fig:colouredtree}
\end{figure}

As usual let $e_0$ be the edge of $\T$ with adjacent regions labelled by $(\bf a, \bf b)$ and let $ q^+(e_0)$ and $ q^-(e_0)$ denote the vertices at the two ends of $e_0$, chosen so that the neighbouring regions are $(\bf a, \bf b, \bf {ab})$ and $(\bf a,\bf b, \bf {ab}^{-1})$ respectively. 
Removing either of these two vertices disconnects $\T$. We deal first with 
the subtree  $\T^+$  consisting of the connected component of $\T \setminus \{q^-(e_0)\}$ which contains $q^+(e_0)$. Note that the regions adjacent to all edges of $\T^+$ correspond to non-negative fractions.  

Let $e$ be a given edge of $\T^+$ and let $q^+(e)$ denote the vertex of $e$ furthest from $q^-(e_0)$.    Let $\g = \g(e)$ be the  unique shortest edge path joining $q^+(e)$ to $q^-(e_0)$, hence including both $e$ and $e_0$. The \emph{coloured level} of $e$, denoted $col.lev(e)$,  is the number of edges $e'$  including $e$ itself in $\g(e)$  with $col(e') =col(e)$. 
Note that $\g(e)$ necessarily includes $e_0$, and, provided $ e \neq e_0$, one or other of the two edges emanating from $q^+(e_0)$ other than $e_0$.  Thus $col.lev(e) = 1$ for all three edges meeting  $q^+(e_0)$ while for all other edges of $\T^+$ we have $col.lev(e) > 1$.  

Now suppose that $e$ is the edge of $\T^+$ whose neighbouring regions are labelled by the given generator pair $(u,v)$.
The proof will be by induction on  $col.lev(e)$.

Suppose first $col.lev(e)=1$. If $e = e_0$ the result is clearly true, since the pair $(a,b)$  is palindromic with respect to itself.
The other two edges emanating from $q^+(e_0)$ have neighbouring regions corresponding to the base pairs $(a,ab)$ and $(ab,b)$, each of which pair is palindromic  with respect to itself, proving the claim.

Suppose the result is proved for all edges of coloured level $k \geq 1$. Let $e$ be an edge  whose adjacent generators are of type 
$(\eta, \eta')$. Suppose that $col(e) =c$ and let  $e'$ be the next  edge of $\g$   with $col(e') = c$ along the path $\g(e)$ from $q^+(e)$ to $q^-(e_0)$.  (Note that such $e'$ always exists  since $k+1\geq 2$.)
By the induction hypothesis the  standard generator pair  $(u,v)$ adjacent to $e'$ is conjugate to a positive pair   $(u',v')$ which is palindromic of the same type   $(\eta, \eta')$. 
 
Let $q^+(e')$ be the vertex of $e'$ closest to $e$, so that 
the subpath path $\g'$ of $ \g$ from $q^+(e')$ to $q^-(e)$ contains no other edges of  colour $c$, where $q^-(e)$ is the vertex of $e$ other than $q^+(e)$. Since there cannot be two adjacent edges of the same colour, the edges of $\g'$ must alternate between the two other colours.  This implies (see Figure~\ref{fig:colouredtree}) that $\g'$ forms part of the boundary of a complementary region $R$ of $\T^+$. Moreover the third edge at each vertex along $\dd R$ (that is, the one which is not contained in $\dd R$),  is coloured $c$.

\begin{figure}
\includegraphics[width=7.5cm]{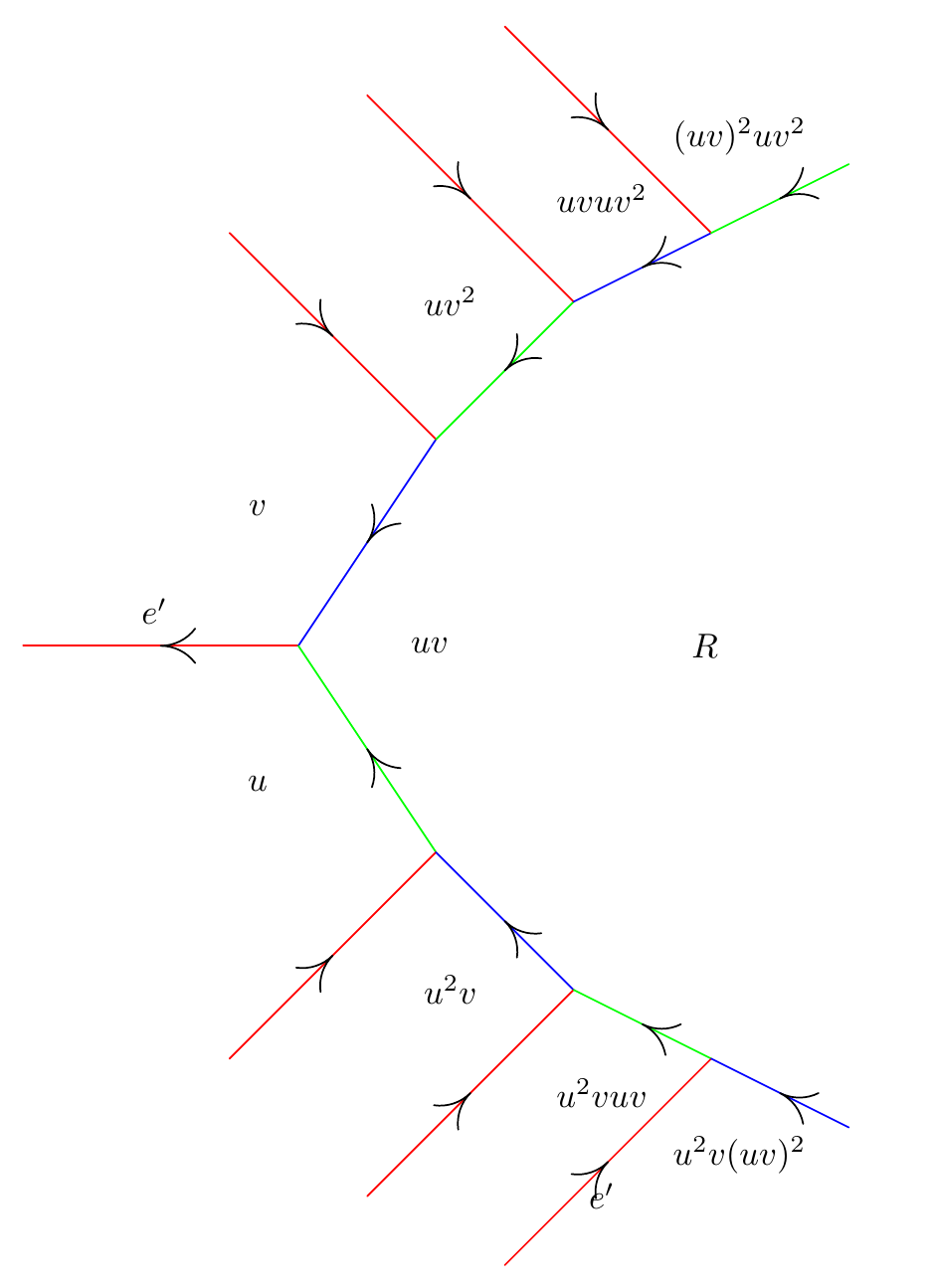}
\caption{Labels of regions round $\dd \R$ showing the $W$-arrows. Note that labels are concatenated in anticlockwise order round the boundary circle.}\label {fig:regionboundary}
\end{figure}

 Without loss of generality, suppose that $u$ is before $v$ in the anti-clockwise order round $\dd \mathbb D$.
  Then the generator associated to $R$ is $uv$. 
  Since $(u,v)$ is a  standard positive pair, moving in anticlockwise order around $\dd \R$ starting from $v$, successive regions have  labels  $$v,u, u^2v, u^2vuv, \ldots, u^2v(uv)^n, \ldots ,$$ see Figure~\ref{fig:regionboundary}. Any successive pair, in particular the pair adjacent to $e$, can be simultaneously conjugated to the form $(u v(uv)^ku, (uv(uv)^{k+1}u)$ for some $k \geq 0$.  Since by hypothesis the generator pair $(u,v)$   is conjugate to a pair $(u',v')$ palindromic with respect to  $(\eta, \eta')$, so is $(u v(uv)^ku, (uv(uv)^{k+1}u)$.
 
 Similarly, the  regions moving clockwise around $\dd R$ starting from $u$ have standard labels   $u,v, uv^2, uvuv^2, \ldots, (uv)^nuv^2, \ldots $. Thus any successive pair can be simultaneously conjugated into the form $(v(uv)^kuv , v(uv)^{k+1}uv)$ for some $k \geq 0$  which 
 is likewise conjugate to a pair palindromic with respect to  $(\eta, \eta')$.

By the same  argument for the tree $ \T^-$ consisting of the connected component of $\T \setminus \{q^+(e_0)\}$ which contains $q^-(e_0)$ we arrive at the statement that the generators associated to each edge  of $ \T^-$ can be written in a form which is palindromic with respect to one of the three generator pairs associated to the edges emanating from $q^-(e_0)$, that is,
$(a, b^{-1})$, $(a, b^{-1}a)$ or   $(b^{-1}a, b^{-1})$. The first pair is obviously palindromic with respect to $(a, b^{-1})$. 
Noting that   $b^{-1}a = (b^{-1}a^{-1})a^2$ which is conjugate to the word $a(b^{-1}a^{-1})a $ palindromic with respect to $(a,ab)$, and  that $b^{-1}a =  b^{-1}  (ab)  b^{-1}  $ which is   palindromic with respect to $(b,ab)$,
the result follows.
 
Now we prove the existence part of the first claim. Suppose that $u \in \P$ is of type $\w \in \EE$ and that $\w' \neq \w$. Choose a generator $v$ of type $\w'$ so that $(u,v)$ is a positive generator pair. By the above there is a conjugate pair $(u',v')$ palindromic with respect to $(\w,\w')$ and $u'$ is a generator as required. 

To see that $u'$ is unique, 
suppose that cyclically shortest positive primitive elements $u$ and $u'$ are in the same  conjugacy class and are both palindromic with respect to the same pair of generators, which we may as well take to be $\{0/1,1/0\}$. Notice that $u$ necessarily has odd length, for otherwise the exponents of $a$ and $b$ are both even. 

Let $u = e_{r}  \ldots e_{1} f   e_{1} \ldots e_{r}$ and  suppose that $f' = e_{k}$ is the centre point about which $u'$ is palindromic  for some $1 \leq k \leq r$.   Then 
$\ldots uu \ldots $ is periodic with minimal period of length $2r+1$ and contains the subword $$  e_{r}  \ldots e_{1} f   e_{1} \ldots e_{{k-1}}f' e_{{k-1}} \ldots e_{1}  f e_{1}\ldots e_{r}$$    so after $f   e_{1} \ldots e_{{k-1}}f' e_{{k-1}} \ldots e_{1}$ the sequence repeats. Since this subword has length $2k<2r+1$ this contradiction proves the result.   

The claimed uniqueness of generator pairs follows immediately.
\end{proof}

 \subsection{Direct proof of Theorem~\ref{introthmB}}\label{sec:direct}
It may also be of interest to give a  direct proof that $\rho \in \B$ implies that $\rho $ has the bounded intersection property. 
Theorem~\ref{direct} below is a simplified version of the proof of this result from~\cite{serPS}. It is based on estimating the distance between  pairs of palindromic axes along their common perpendicular. We use the estimate~\eqref{distbound} derived from  the invariance of the amplitude (up to sign) under change of generators to improve the corresponding estimate  in Proposition 4.6 in~\cite{serPS}.

\begin{theorem}\label{direct}(Direct proof of Theorem~\ref{introthmB}.)
If $\rho \in \B$ then $\rho $ has the bounded intersection property. 
\end{theorem}
\begin{proof}

Assume that  $\rho \in \B$ and choose $M_0 \geq 2$ and a finite connected non-empty subtree tree $T_F$ of $\T$ as in Theorem~\ref{sinktree}. Let $\Omega(T_F) $ be the set of regions $\bu \in \Omega$ such that $\bu$ is adjacent to an edge of $T_F$.
By enlarging $T_F$ if necessary, we can ensure that every region in $\Omega(2)$ is adjacent to some edge of $T_F$.
In addition, since there are only finitely many possible pairs of elements of $\Omega(2)$, we may yet further enlarge $T_F$ so that  no edge outside $T_F$ is adjacent to a region  in $\Omega(2)$ on both sides.

 Suppose the generator $u = u_1$ is palindromic with respect $ \w$ and  that $\w' \neq \w$. Without loss of generality, we may take $u$ positive.  Let $\E = \E_{\w,\w'}$ be the corresponding special hyperelliptic axis.
Let $\Xi$ denote the set of axes corresponding to palindromic representatives of $\bv \in \Omega(T_F)$ which are of types either $\w$ or $\w'$. It is sufficient to see that   $\Ax U$ meets $\E$ at a uniformly bounded distance to one of the finitely many axes in $\Xi$.

 If $\bu_1 \in \Omega(T_F)$ there is nothing to prove, so suppose that $\bu_1 \notin \Omega(T_F)$.    Choose an  oriented edge $\vec e_1$  in $\dd \bu_1$.  Then there is a strictly descending path $\beta$ of $T$-arrows   $\vec e_1, \ldots, \vec e_n$ so that  the head of $\vec e_n$ meets an edge in $T_F$, and this is the first edge in $\b$ with this properly.
  We claim that  there is a sequence of positive generators
$u_1 = u, u_2, \ldots,  u_k \in \P$ such that for $i = 1, \ldots, k-1$:
\begin{enumerate}
\item $(\bu_i, \bu_{i+1})$ are neighbours adjacent to an edge  of $\beta$.
\item $u_i \in \bu_i, u_{i+1} \in \bu_{i+1}$ and 
$(u_i, u_{i+1})$   is a positive generator pair palindromic  with respect to $(\eta, \eta')$.
\item $\bu_k \in \Omega(T_F)$  but  $\bu_i \notin \Omega(T_F),   1 \leq i <k$.
 \end{enumerate}

Suppose that $u_1 , \ldots, u_i$ have been constructed with   properties (1) and (2) with $i \geq 1$ and that $\bu_i \notin \Omega(T_F)$. The path $\b$ travels round $\dd \bu_i$,  eventually leaving it along an arrow $\vec e$ which points out of $\dd \bu_i$. If $\bu_i$ is of type $\w$ (respectively $\w'$) then of the two regions adjacent to  $\Vec e$, one, $\bu'$ say, is of type $\w'$ (respectively $\w$). Set $\bu_{i+1}= \bu'$ and choose $u_{i+1} \in \bu_{i+1}$ so that   $(u_i, u_{i+1})$ is positive and palindromic with respect to $(\w, \w')$. (Notice that we are using the uniqueness of the palindromic form for $u_i$, in other words if $(u_{i-1}, u_{i})$ is the positive  palindromic pair associated to the regions 
 $(\bu_{i-1}, \bu_{i})$ then  $(u_{i}, u_{i+1})$ is the positive  palindromic pair associated to the regions 
 $(\bu_{i}, \bu_{i+1})$.)   If $\bu_{i+1} \in \Omega(T_F)$ we are done, otherwise continue as before.  Since $\b$ eventually lands on an edge of $T_F$, the process terminates. This proves the claim.

 Since $(u_i, u_{i+1})$ are palindromic  with respect to $(\eta, \eta')$, the axes $\Ax U_i, \Ax U_{i+1}$ are orthogonal to the hyperelliptic axis $\E_{\w,\w'}$  and hence Equation~\eqref{distbound}   gives $d(\Ax U_i, \Ax U_{i+1}) \leq O(e^{-\ell(U_i)}),   1 \leq i <k$. 

Now let $\Vec e$ be the oriented edge between $\bu_{k-1}, \bu_k$   and let  $\W(\Vec e)$ be its wake.  Then since the edge between  $\bu_i, \bu_{i+1}$ is always oriented towards $\Vec e$, we see that   $\bu_i \in \W(\Vec e), 0 \leq i \leq k$. 
  Let $F_{\vec {e}}$ be the 
 Fibonacci function   on $\W(\Vec e)$   defined  immediately above Lemma~\ref{fibonacciwake}. 
It is not hard to see that for $0 \leq i \leq k$ we have $F_{\vec {e}}(\bu_i) \geq k-i$.
By construction, $\bu_{k-1} \notin \Omega(T_F) $ so that, by our assumption on $T_F$, we have   $\bu_{k-1} \notin \Omega(2)$.  
Moreover  by connectivity of $T_F$, no edge in $\W_{\E}(\Vec e)$  is in $T_F$ and hence none of these edges is adjacent on both sides to regions  in $\Omega(2)$. 
Thus by Lemma~\ref{fibonacciwake}, there exist $c>0, n_0 \in \NN$ depending only on $\rho$ and not on $\Vec e$ such that $\log^+{|\Tr U_i|} \geq c (k-i)$ for all but at most $n_0$   of  the regions $\bu_i$.

 Hence for all except some uniformly bounded number of the regions $\bu_i$, $ \ell(U_i)   \geq c(k-i) -\log 2$. Since all axes $\Ax U_i$  intersect  $\E$ orthogonally in points $P_i$ say, it follows that 
  $d(\Ax U_1, \Ax U_k) $ is bounded above by the sum $\sum_1^{k-1} d(\Ax U_i, \Ax U_{i+1}) $ of the distances between the points $P_i, P_{i+1}$. Since $d(\Ax U_i, \Ax U_{i+1}) \leq O(e^{-\ell(U_i)}),   1 \leq i <k$, 
   the distance from $ \Ax U_1$  to one of the finitely many axes in $\Xi$ is uniformly bounded above, and we are done.
\end{proof}


\begin{thebibliography}{000}
 
 
   \bibitem{BSeries}
J. Birman and C. Series.
\newblock  Geodesics with multiple self-intersections and 
symmetries on Riemann surfaces.
\newblock In {\em Low dimensional topology and Kleinian 
groups},   D. Epstein  ed., LMS Lecture Notes 112, Cambridge Univ. Press, 3 -- 12,   1986.

\newblock {\em Proc. London Math. Soc.  77}, 697--736, 1998.

   
  
   \bibitem{bow_mar}
B.~H. Bowditch.
\newblock {M}arkoff triples and quasi-{F}uchsian groups.
\newblock {\em Proc. London Math. Soc.  77}, 697--736, 1998.




   \bibitem{CMZ}
M.~Cohen, W.~Metzler and A.~Zimmermann.
\newblock Enumerating palindromes and primitives in rank two free groups.
\newblock {\em Math. Annalen}, 257, 435-- 446, 1981.
 


 \bibitem{Fen}
W.~Fenchel.
\newblock {\em Elementary geometry in hyperbolic space.}
\newblock De Gruyter Studies in Mathematics,  Vol. 11, 1989.



   \bibitem{gilmankeen1}
J.~Gilman and L.~Keen.
\newblock Enumerating palindromes and primitives in rank two free groups.
\newblock {\em Journal of algebra}, 332, 1--13, 2011.

 \bibitem{gilmankeen2}
J.~Gilman and L.~Keen.
\newblock Discreteness criteria and the hyperbolic geometry of palindromes.
\newblock {\em Conformal geometry and dynamics}, 13, 76 --90, 2009.



\bibitem{goldman}
W. Goldman. 
\newblock The modular group action on real $SL(2)$-characters of a one-holed torus.
\newblock {\em Geometry and Topology}  7, 443 -- 486, 2003.


\bibitem{goldman2}
W. Goldman. 
\newblock Trace coordinates on Fricke spaces of some
simple hyperbolic surfaces.
\newblock In {\em Handbook of Teichm\"uller theory Vol. II},  IRMA Lect. Math. Theor. Phys., 13, Euro. Math. Soc., Z\"urich, 611-- 684, 2009.


 
\bibitem{ksriley}   L. Keen and C. Series. 
\newblock The Riley slice of Schottky space. 
\newblock {\em Proc. London Math. Soc.}, 69,  72 -- 90, 1994.



\bibitem{LX} J. Lee and B. Xu.
 \newblock Bowditch's Q-conditions and Minsky's primitive stability.
 \newblock {\em Trans. AMS}, 373,  1265 -- 1305, 2020.
 

 
 
\bibitem{lupi}   D. Lupi. 
\newblock Primitive stability and Bowditch conditions for rank 2 free group representations. 
\newblock   {\em Thesis, University of Warwick},  2016.



\bibitem{Minsky}   Y. Minsky. 
\newblock On dynamics of $Out(F_n)$ on $PSL(2,\CC)$ characters. 
\newblock {\em Israel Journal of Mathematics}, 193,  47 -- 70, 2013.


 
\bibitem{serwolp}
C.~Series.
\newblock  An extension of Wolpert's derivative formula.
\newblock {\em Pacific J. Math.}, 
197, 223 -- 239, 2001.

\bibitem{serPS}
C.~Series.
\newblock  Primitive stability and Bowditch's BQ-conditions are equivalent.
\newblock {\em arXiv:2530070 [math.GT]},  2019.


\bibitem{serInt}
C.~Series.
\newblock  The Geometry of Markoff Numbers. 
 \newblock {\em Math. Intelligencer},  7, 20 -- 29, 1985.
   
\bibitem{sty}
 C.~Series, S.P.~Tan, Y.~Yamashita.
\newblock The diagonal slice of Schottky space. 
\newblock {\em Algebraic and Geometric Topology}, 17, 2239 -- 2282,  2017.  



\bibitem{tan_gen}
S.P.  Tan, Y. L.  Wong  and Y. Zhang.
\newblock Generalized {M}arkoff maps and {M}c{S}hane's identity.
\newblock {\em Adv. Math.  217}, 761--813, 2008.

  


\end{thebibliography}
\end{document}